\documentclass[a4paper,11pt,leqno]{article}

\usepackage[latin1]{inputenc}
\usepackage[T1]{fontenc} 
\usepackage[english]{babel}
\usepackage{amsmath}
\usepackage{amsfonts}
\usepackage{amsthm}
\usepackage{amssymb}
\usepackage{mathrsfs}
\usepackage{latexsym}
\usepackage{graphicx}
\usepackage{bm}

\theoremstyle{definition}
\newtheorem{definiz}{Definition}[section]

\newtheorem{oss}[definiz]{Remark}

\theoremstyle{plane}
\newtheorem{teorema}[definiz]{Theorem}

\newtheorem{lemma}[definiz]{Lemma}

\newcommand{\mbb}{\mathbb}

\newcommand{\mc}{\mathcal}
\newcommand{\veps}{\varepsilon}
\newcommand{\what}{\widehat}
\newcommand{\wtilde}{\widetilde}
\newcommand{\vphi}{\varphi}

\title{\Large{\bfseries{A NOTE ON NON-HOMOGENEOUS HYPERBOLIC OPERATORS WITH LOW REGULARITY COEFFICIENTS}}}

\author{FERRUCCIO COLOMBINI\\
\small{\textit{Dipartimento di Matematica, Universit\`a di Pisa}}\\
\small{\textit{Largo B. Pontecorvo 5, 56127 Pisa, ITALY}}\\
\small{\ttfamily{colombini@dm.unipi.it}} \vspace{0.5cm}\\
FRANCESCO FANELLI\\
\small{\textit{SISSA}}\\
\small{\textit{via Beirut 2/4, 34151 Trieste, ITALY}}\\
\small{\ttfamily{francesco.fanelli@sissa.it}}}

\date{February 16, 2010}

\begin{document}

\maketitle

\subsubsection*{Abstract}
In this paper we obtain an energy estimate for a complete strictly hyperbolic operator with second order coefficients satisfying a
log-Zygmund-continuity condition with respect to $t$, uniformly with respect to $x$, and a log-Lipschitz-continuity condition with respect to $x$,
uniformly with respect to $t$.

\subsubsection*{Keywords}
Hyperbolic equations, log-Zygmund regularity, loss of derivatives.
 
\subsubsection*{MS Classification 2000}
35L15

\section{Introduction}
Let us consider the second order operator
\begin{equation} \label{eq:P}
 P\,=\,\partial^2_t\,-\,\sum_{i,j=1}^n\partial_{x_i}(a_{ij}(t)\partial_{x_j})
\end{equation}
and suppose that $P$ is strictly hyperbolic, i.e. there exist two positive constants $\lambda_0\leq\Lambda_0$ such that
\begin{equation} \label{eq:P_elliptic}
 \lambda_0\,|\xi|^2\,\leq\,\sum_{i,j=1}^n a_{ij}(t)\,\xi_i\xi_j\,\leq\,\Lambda_0\,|\xi|^2
\end{equation}
for all $t\in\mbb{R}$ and all $\xi\in\mbb{R}^n$.

It is well-known (see e.g.~\cite{Horm} and~\cite{Treves}) that, if the coefficients $a_{ij}$ are Lipschitz-continuous, then the following energy 
estimate holds for the operator $P$: for all $s\in\mbb{R}$, there exists $C_s>0$ such that
\begin{eqnarray} \label{eq:en-estim_t_Lip}
& & \sup_{t\in[0,T]}\left(\|u(t,\cdot)\|_{H^{s+1}}\,+\,\|\partial_tu(t,\cdot)\|_{H^s}\right)\,\leq \\
& & \qquad\qquad\qquad\qquad\qquad\leq\,C_s\,\left(\|u(0,\cdot)\|_{H^{s+1}}\,+\,\|\partial_tu(0,\cdot)\|_{H^s}\,+\,
\int^T_0\|Pu(t,\cdot)\|_{H^s}\,dt\right) \nonumber
\end{eqnarray}
for every function $u\in\mc{C}^2([0,T];H^\infty(\mbb{R}^n))$.\\
In particular, the previous energy estimate implies that the Cauchy problem for (\ref{eq:P}) is well-posed in the space $H^\infty$, with no loss
of derivatives.

On the contrary, if the coefficients $a_{ij}$ are not Lipschitz-continuous, then (\ref{eq:en-estim_t_Lip}) is no more true in general, as it is shown
by an example given by Colombini, De Giorgi and Spagnolo in the paper~\cite{Colomb-DeG-S}. Nevertheless,
under suitable weaker regularity assumptions on the coefficients, one can recover the $H^\infty$-well-posedness again, but this time
from an energy estimate with loss of derivatives.

A first result of this type was obtained in the quoted paper~\cite{Colomb-DeG-S}. The authors supposed that there
was a constant $C>0$ such that, for all $\veps\in]0,T]$,
\begin{equation} \label{eq:cont_t_Log-Lip_int}
 \int^{T-\veps}_0|a_{ij}(t+\veps)-a_{ij}(t)|\,dt\,\leq\,C\,\veps\,\log\left(1+\frac{1}{\veps}\right)\,.
\end{equation}
The Fourier trasform with respect to $x$ of the equation, togheter with the new ``approximate energy technique'' 
(i.e. the approximation of the coefficients is different in different zones of the phase space), enabled them to obtain the following energy estimate:
there exist strictly positive constants $K$ (indipendent of $s$) and $C_s$ such that
\begin{eqnarray} \label{eq:en-estim_t_Log-Lip}
& & \sup_{t\in[0,T]}\left(\|u(t,\cdot)\|_{H^{s+1-K}}\,+\,\|\partial_tu(t,\cdot)\|_{H^{s-K}}\right)\,\leq \\
& & \qquad\qquad\qquad\qquad\qquad\leq\,C_s\,\left(\|u(0,\cdot)\|_{H^{s+1}}\,+\,\|\partial_tu(0,\cdot)\|_{H^s}\,+\,
\int^T_0\|Pu(t,\cdot)\|_{H^s}\,dt\right) \nonumber
\end{eqnarray}
for all $u\in\mc{C}^2([0,T];H^\infty(\mbb{R}^n))$.

Considering again the case that the coefficients of $P$ depend only on the time variable, in the recent paper~\cite{Tarama}
(see also~\cite{Yama}) Tarama has weakened the regularity hypotesis further, supposing a log-Zygmund type integral condition, i.e.
that there exists a constant $C>0$ such that, for all $\veps\in]0,T/2]$,
\begin{equation} \label{eq:cont_t_Log-Zyg_int}
 \int^{T-\veps}_{\veps}|a_{ij}(t+\veps)+a_{ij}(t-\veps)-2a_{ij}(t)|\,dt\,\leq\,C\,\veps\,\log\left(1+\frac{1}{\veps}\right)\,.
\end{equation}
Nevertheless, he has been able to prove the well-posedness to the Cauchy problem for (\ref{eq:P}) in the space $H^\infty$: the improvement with respect 
to~\cite{Colomb-DeG-S} was obtained introducing a new type of approximate energy, which involves the second derivatives
of the approximating coefficients.

Much more difficulties arise if the operator $P$ has coefficients depending both on the time variable $t$ and on the space variables $x$. This
case was considered by Colombini and Lerner in the paper~\cite{Colomb-Lern}. They supposed a pointwise log-Lipschitz regularity condition, i.e.
that there exists $C>0$ such that, for all $\veps\in]0,T]$,
\begin{equation} \label{eq:cont_t-x_Log-Lip_punt}
 \sup_{\begin{array}{c}
        \scriptstyle y,z\in[0,T]\times\mbb{R}^n \\
	\scriptstyle |z|=\veps
       \end{array}} |a_{ij}(y+z)-a_{ij}(y)|\,\leq\,C\,\veps\,\log\left(1+\frac{1}{\veps}\right)\,.
\end{equation}
Because the coefficients of the operator $P$ depend also on the space variables, here the Littlewood-Paley dyadic decomposition with respect
to $x$ takes the place of the Fourier trasform, and it is, togheter with the approximate energy technique, the key tool to obtain the energy estimate:
for all fixed $\theta\in]0,1/4]$, there exist $\beta$, $C>0$ and $T^*\in]0,T]$ such that
\begin{eqnarray} \label{eq:en-estim_t-x_Log-Lip}
& & \sup_{t\in[0,T^*]}\left(\|u(t,\cdot)\|_{H^{-\theta+1-\beta t}}\,+\,\|\partial_tu(t,\cdot)\|_{H^{-\theta-\beta t}}\right)\,\leq \\
& & \qquad\qquad\qquad\qquad\qquad\leq\,C\,\left(\|u(0,\cdot)\|_{H^{-\theta+1}}\,+\,\|\partial_tu(0,\cdot)\|_{H^{-\theta}}\,+\,
\int^T_0\|Pu(t,\cdot)\|_{H^{-\theta-\beta t}}\,dt\right) \nonumber
\end{eqnarray}
for all $u\in\mc{C}^2([0,T^*];H^\infty(\mbb{R}^n))$.\\
In this case, the loss of derivatives gets worse with the increasing of time.

In a recent paper (\hspace{-0.1cm}~\cite{Colomb-DelSant_2}), Colombini and Del Santo considered the case of one space variable (i.e. $n=1$) and
studied again the case of the coefficient $a$ depending both on $t$ and $x$, but under a special regularity condition: they mixed condition
(\ref{eq:cont_t_Log-Zyg_int}) togheter with (\ref{eq:cont_t-x_Log-Lip_punt}). In particular, they supposed $a$ to be log-Zygmund-continuous
with respect to $t$, uniformly with respect to $x$, and log-Lipschitz-continuous with respect to $x$, uniformly with respect to $t$. The dyadic
decomposition technique and the Tarama's approximate energy enabled them to obtain an estimate similar to (\ref{eq:en-estim_t-x_Log-Lip}).

The reason why they focused on the case $n=1$ is that the case of several space variables needs some different and new ideas in the definition of the
microlocal energy: this point still remains as an open problem.

\bigskip
In the present note, we will consider the case of the non-homogeneous operator
\begin{equation} \label{eq:L}
Lu\,=\,\partial_{t}^{2}u\,-\,\partial_x(a(t,x)\partial_x u)\,+\,b_0(t,x)\partial_tu\,+\,b_1(t,x)\partial_x u\,+\,c(t,x)u\;,
\end{equation}
where the coefficient $a$ satisfy the same regularity assumptions as in~\cite{Colomb-DelSant_2}. We will also suppose that $b_0$,
$b_1\in L^\infty(\mbb{R}_t;\mc{C}^{\omega}(\mbb{R}_x))$, $\omega>0$, and $c$ bounded on $\mbb{R}_t\times\mbb{R}_x$.
We will apply the Littlewood-Paley decomposition and the Tarama's approximate energy again to obtain an energy estimate with a loss of
derivatives that depends on $t$, as in (\ref{eq:en-estim_t-x_Log-Lip}).

One can find the estimate of the second order coefficient $a$ in the paper~\cite{Colomb-DelSant_2}, however, for reader's convenience,
we will give here all the details.
 
\section{Main result}
Let $a:\mbb{R}^2\rightarrow\mbb{R}$ be a function such that, for positive constants $\lambda_0\leq\Lambda_0$ and $C_0$, one has, for all
$(t,x)\in\mbb{R}^2$ and all $\tau>0$, $y>0$,
\begin{eqnarray}
 & \lambda_0\leq a(t,x)\leq\Lambda_0   \label{(N)eq_condiz_1.1} \\
 & \sup_{(t,x)}|a(t+\tau,x)+a(t-\tau,x)-2a(t,x)|\leq C_0\,\tau\log\left(\frac{1}{\tau}+1\right) \label{(N)eq_condiz_1.2} \\
 & \sup_{(t,x)}|a(t,x+y)-a(t,x)|\leq C_0\,y\log\left(\frac{1}{y}+1\right) \label{(N)eq_condiz_1.3}
\end{eqnarray}
Moreover, let
\begin{equation} \label{(N)eq_condiz_B}
b_0\,,\,b_1\;\in\,L^\infty(\mbb{R}_t;\mc{C}^{\omega}(\mbb{R}_x))\,,
\end{equation}
where $\omega>0$, and
\begin{equation} \label{(N)eq_condiz_c}
c\;\in\,L^{\infty}(\mbb{R}_t\times\mbb{R}_x)\;.
\end{equation}

\begin{teorema} \label{(N)teor_teorema1}
Let us consider, on the whole space $\mbb{R}^2$, the operator
\begin{equation} \label{(N)operatore}
Lu\,=\,\partial_{t}^{2}u\,-\,\partial_x(a(t,x)\partial_x u)\,+\,b_0(t,x)\partial_tu\,+\,b_1(t,x)\partial_x u\,+\,c(t,x)u\;,
\end{equation}
where the coefficients $a$, $b_0$, $b_1$ and $c$ satisfy the hypothesis (\ref{(N)eq_condiz_1.1})-(\ref{(N)eq_condiz_c}).\\
Then, for all fixed
$$
\theta\,\in\,\left]0,\min\left\{\frac{1}{2},\frac{\omega}{1+\log2}\right\}\right[\;,
$$
there exist $\beta^*>0$, a time $T\in\mbb{R}$ and a constant $C>0$ such that
\begin{eqnarray} \label{(N)eq_tesi}
 & & \sup_{[0,T]}\left(\|u(t,\cdot)\|_{H^{1-\theta-\beta^*t}}+
\|\partial_t u(t,\cdot)\|_{H^{-\theta-\beta^*t}}\right)\,\leq \\
 & & \qquad\qquad\qquad\qquad\leq\,C\,\left(\|u(0,\cdot)\|_{H^{1-\theta}}+
\|\partial_t u(0,\cdot)\|_{H^{-\theta}}+\int_{0}^{T}\|Lu(t,\cdot)\|_{H^{-\theta-\beta^*t}}dt\right) \nonumber
\end{eqnarray}
for all $u\in\mc{C}^2([0,T];H^\infty(\mbb{R}_x))$.
\end{teorema}

\section{Proof of Theorem \ref{(N)teor_teorema1}}
\subsection{Approximation of the coefficient $a(t,x)$}
Let $\rho\in\mc{C}_{0}^{\infty}(\mbb{R})$ be an even function such that:
\begin{enumerate}
 \item $0\leq\rho\leq1$
 \item $supp\,\rho\,\subset\,[-1,1]$
 \item $\int\rho(s)ds=1$
 \item $|\rho'(s)|\leq2$;
\end{enumerate}
for all $0<\veps\leq1$, we set $\rho_{\veps}(s)=\frac{1}{\veps}\,\rho\left(\frac{s}{\veps}\right)$.

Then, for all $0<\veps\leq1$, we define
\begin{equation} \label{(N)eq_a_eps}
a_\veps(t,x)\,:=\,\int_{\mbb{R}_t\times\mbb{R}_x}\rho_\veps(t-s)\,\rho_\veps(x-y)\,a(s,y)\,dsdy\;.
\end{equation}

\begin{lemma} \label{(N)lemma_approx}
The following inequalities hold true:
\begin{enumerate}
 \item for all $\veps\in]0,1]$, for all $(t,x)\in\mbb{R}^2$, one has
\begin{equation} \label{(N)eq_condiz_1.1_eps}
\lambda_0\,\leq\,a_\veps(t,x)\,\leq\,\Lambda_0\;;
\end{equation}
 \item for all $\veps\in]0,1]$, one has
\begin{equation} \label{(N)eq_a_eps-a}
\sup_{(t,x)}|a_\veps(t,x)-a(t,x)|\,\leq\,
\frac{3}{2}C_0\,\veps\,\log\left(\frac{1}{\veps}+1\right)\;;
\end{equation}
 \item for all $\sigma\in]0,1[$, a constant $c_\sigma>0$ exists such that, for all $\veps\in]0,1]$,
\begin{equation} \label{(N)eq_d_t_a_eps}
\sup_{(t,x)}|\partial_t a_\veps(t,x)|\,\leq\,c_\sigma (\Lambda_0+C_0)\,\veps^{\sigma-1}\;;
\end{equation}
 \item for all $\veps\in]0,1]$, one has
\begin{eqnarray}
\sup_{(t,x)}|\partial_x a_\veps(t,x)| & \leq &
C_0\|\rho'\|_{L^1}\log\left(\frac{1}{\veps}+1\right) \label{(N)eq_d_x_a_eps} \\
\sup_{(t,x)}\left|\partial_{t}^{2} a_\veps(t,x)\right| & \leq &
\frac{C_0}{2}\|\rho''\|_{L^1}\frac{1}{\veps}\log\left(\frac{1}{\veps}+1\right) \label{(N)eq_d_t^2_a_eps} \\
\sup_{(t,x)}|\partial_t\partial_x a_\veps(t,x)| & \leq &
C_0\|\rho'\|_{L^1}^{2}\frac{1}{\veps}\log\left(\frac{1}{\veps}+1\right) \label{(N)eq_d_t_d_x_a_eps}\;.
\end{eqnarray}
\end{enumerate}
\end{lemma}

\begin{proof}
Inequalities in (\ref{(N)eq_condiz_1.1_eps}) immediately follow from the fact that $|\rho|\leq1$.

Relation (\ref{(N)eq_a_eps-a}), instead, follows from (\ref{(N)eq_condiz_1.2}), after one has observed that
$$
a_\veps(t,x)-a(t,x)=\frac{1}{2}\int_{\mbb{R}_t}\rho_\veps(s)\int_{\mbb{R}_y}\rho_\veps(x-y)(a(t+s,y)+a(t-s,y)-2a(t,y))\,dy\,ds\;,
$$
where we have used the fact that $\rho$ is an even function. 

Moreover, one has
$$
\partial_t^2a_\veps(t,x)=\frac{1}{2}\int\rho_\veps''(s)\int\rho_\veps(x-y)(a(t+s,y)+a(t-s,y)-2a(t,y))\,dy\,ds\;,
$$
from which one can deduce (\ref{(N)eq_d_t^2_a_eps}).

Inequalities (\ref{(N)eq_d_x_a_eps}) and (\ref{(N)eq_d_t_d_x_a_eps}) derive from (\ref{(N)eq_condiz_1.3}) in a very similar way.

Finally, relation (\ref{(N)eq_d_t_a_eps}) is a consequence of the fact that (\ref{(N)eq_condiz_1.1}) and (\ref{(N)eq_condiz_1.2}) imply that
for all $\sigma\in]0,1[$, a constant $c_\sigma'>0$ exists such that, for all $\tau>0$, one has
\begin{equation} \label{(N)eq_2.6}
\sup_{(t,x)}|a(t+\tau,x)-a(t,x)|\,\leq\,c_\sigma'\,(\Lambda_0+C_0)\,\tau^\sigma\;.
\end{equation}
\end{proof}

\subsection{Littlewood-Paley decomposition} \label{sub-sec:L-P_dec}
We collect here some well-known facts about dyadic decomposition, referring to~\cite{Bony},~\cite{Colomb-Lern} and~\cite{Metivier} for the details.

Let $\vphi_0\in\mc{C}_{0}^{\infty}(\mbb{R}_\xi)$ be an even function, decreasing on $[0,+\infty[$, such that $0\leq\vphi_0\leq1$ and
$$
\vphi_0(\xi)=1\,\mbox{ if }\,|\xi|\leq1\quad\mbox{,}\quad\vphi_0(\xi)=0\,\mbox{ if }\,|\xi|\geq2.
$$
We set $\vphi(\xi)\,=\,\vphi_0(\xi)-\vphi_0(2\xi)$ and, for $\nu\in\mbb{N}\backslash\{0\}$, $\vphi_{\nu}(\xi):=\vphi(2^{-\nu}\xi)$.

For a tempered distribution $u\in H^{-\infty}(\mbb{R})$, we define
$$
u_\nu(x)\,:=\,\vphi_\nu(D_x)u(x)\,=\,\frac{1}{2\pi}\int e^{ix\xi}\vphi_\nu(\xi)\,\what{u}(\xi)\,d\xi\,=\,
\frac{1}{2\pi}\int\what{\vphi_\nu}(y)u(x-y)\,dy\,;
$$
for all $\nu$, $u_\nu$ is an entire analytic function belonging to $L^2$.\\
Moreover, for all $s\in\mbb{R}$ there exists a constant $C_s>0$ such that
\begin{equation} \label{eq:dyadic-dec_Sobolev-norm}
 \frac{1}{C_s}\,\sum^{+\infty}_{\nu=0}2^{2\nu s}\|u_\nu\|^2_{L^2}\,\leq\,\|u\|^2_{H^s}\,\leq\,
C_s\,\sum^{+\infty}_{\nu=0}2^{2\nu s}\|u_\nu\|^2_{L^2}
\end{equation}
and the following inequalities (called ``Bernstein's inequalities'') hold:
\begin{eqnarray}
 \|\partial_x u_\nu\|_{L^2} & \leq & 2^{\nu+1}\,\|u_\nu\|_{L^2}\qquad\qquad\;\mbox{  for all } \nu\geq0 \label{eq:Bern-ineq_deriv<} \\
 \|u_\nu\|_{L^2} & \leq & 2^{-(\nu-1)}\,\|\partial_x u_\nu\|_{L^2}\qquad\mbox{ for all } \nu\geq1 \label{eq:Bern-ineq_<deriv}\,.
\end{eqnarray}

We end this subsection quoting a result which will be useful in the following; for its proof, see~\cite{Colomb-Lern}.\\
We denote with $[A,B]$ the commutator between two linear operators $A$ and $B$ and with $\mc{L}(L^2)$ the space of bounded linear operators
from $L^2$ to $L^2$.
\begin{lemma} \label{lemma:comm}
 \begin{enumerate}
  \item There exist $C>0$, $\nu_0\in\mbb{N}$ such that, for all $a\in L^\infty(\mbb{R})$ satisfying
$$
\sup_{x\in\mbb{R}}|a(x+y)-a(x)|\,\leq\,C_0\,y\,\log\left(1+\frac{1}{y}\right)
$$
for all $y>0$, one has, for all $\nu\geq\nu_0$,
\begin{equation} \label{eq:comm_Log-Lip}
 \|\,[\vphi_\nu(D_x),a(x)]\,\|_{\mc{L}(L^2)}\,\leq\,C\,(\|a\|_{L^\infty}+C_0)\,2^{-\nu}\,\nu\,.
\end{equation}
 \item There exist $C>0$, $\nu_0\in\mbb{N}$ such that, for all $b\in\mc{C}^\omega(\mbb{R})$ and all $\nu\geq\nu_0$, one has
\begin{equation} \label{eq:comm_Holder}
 \|\,[\vphi_\nu(D_x),b(x)]\,\|_{\mc{L}(L^2)}\,\leq\,C\,\|b\|_{\mc{C}^\omega}\,2^{-\nu}\,.
\end{equation}
 \end{enumerate}
\end{lemma}

\subsection{Approximate and total energy}
Let $T_0>0$ and $u\in\mc{C}^2([0,T_0];H^\infty(\mbb{R}_x))$. If we set $u_\nu(t,x)=\vphi_\nu(D_x)u(t,x)$, we obtain
\begin{eqnarray} \label{(N)eq_ni}
(Lu)_{\nu} & = & \partial_{t}^{2}u_{\nu}\,-\,\partial_x(a(t,x)\partial_x u_{\nu}) - \partial_x([\vphi_{\nu}(D_x),a]\partial_x u) \\
           & + &  b_0(t,x)\partial_t u_\nu + [\vphi_\nu(D_x),b_0]\partial_t u\,+\,
b_1(t,x)\partial_x u_\nu + [\vphi_\nu(D_x),b_1]\partial_x u \nonumber\\
           & + &  c(t,x)u_\nu + [\vphi_\nu(D_x),c]u \nonumber\;.
\end{eqnarray}

Now, we introduce the approximate energy of $u_\nu$ (see~\cite{Colomb-DelSant_2} and~\cite{Tarama}), setting
\begin{equation} \label{(N)eq_e_nu-eps}
e_{\nu,\varepsilon}(t):=\int_{\mbb{R}}\left(\frac{1}{\sqrt{a_\varepsilon}}\left|\partial_t u_\nu+
\frac{\partial_t\sqrt{a_\varepsilon}}{2\sqrt{a_\varepsilon}}u_\nu\right|^2+\sqrt{a_\varepsilon}|\partial_x u_\nu|^2+|u_\nu|^2\right)\,dx\;,
\end{equation}
and, taken $\theta$ as in the hypothesis of Theorem \ref{(N)teor_teorema1}, we define the total energy of $u$:
\begin{equation} \label{(N)eq_E_tot}
E(t)\,:=\,\sum_{\nu=0}^{+\infty}e^{-2\beta(\nu+1)t}\,2^{-2\nu\theta}\,e_{\nu,2^{-\nu}}(t)\;,
\end{equation}
where $\beta>0$ will be fixed later on.

\begin{oss} \label{remark:E(0)-E(t)}
From (\ref{eq:dyadic-dec_Sobolev-norm}) and Bernstein's inequalities (\ref{eq:Bern-ineq_deriv<})-(\ref{eq:Bern-ineq_<deriv}), it's easy to see that
there exist two positive constants $C_\theta$ and $C_\theta'$ such that
\begin{eqnarray*}
E(0) & \leq & C_\theta\,(\|\partial_tu(0,\cdot)\|_{H^{-\theta}}\,+\,\|u(0,\cdot)\|_{H^{1-\theta}}) \\ 
E(t) & \geq & C_\theta'\,(\|\partial_tu(t,\cdot)\|_{H^{-\theta-\beta^*t}}\,+\,
\|u(t,\cdot)\|_{H^{1-\theta-\beta^*t}}) 
\end{eqnarray*}
where we have set $\beta^*=\beta(\log2)^{-1}$.
\end{oss}

\medskip
First, we derive $e_{\nu,\veps}$, defined by (\ref{(N)eq_e_nu-eps}), with respect to the time variable, and we obtain
\begin{eqnarray*}
\frac{d}{dt}e_{\nu,\veps}(t) & = & \int\frac{2}{\sqrt{a_\veps}}\,\,\Re\left(\partial_t^2u_\nu\,\cdot\,
\overline{\left(\partial_tu_\nu+\frac{\partial_t\sqrt{a_\veps}}{2\sqrt{a_\veps}}u_\nu\right)}\right)\,dx \\
 & + & \int\frac{2}{\sqrt{a_\veps}}\,\,\Re\left(R_\veps u_\nu\,\cdot\,
\overline{\left(\partial_tu_\nu+\frac{\partial_t\sqrt{a_\veps}}{2\sqrt{a_\veps}}u_\nu\right)}\right)\,dx \\
 & + & \int\partial_t\sqrt{a_\veps}\,\,|\partial_xu_\nu|^2\,dx \\
 & + & \int2\sqrt{a_\veps}\,\,\Re(\partial_xu_\nu\,\cdot\,\overline{\partial_x\partial_tu_\nu})\,dx \\
 & + & \int2\,\,\Re(u_\nu\,\cdot\,\overline{\partial_tu_\nu})\,dx\;,
\end{eqnarray*}
where $R_\veps v\,=\,\partial_t\left(\frac{\partial_t\sqrt{a_\veps}}{2\sqrt{a_\veps}}\right)v-
\left(\frac{\partial_t\sqrt{a_\veps}}{2\sqrt{a_\veps}}\right)^2v$.
Now, we can put in the previous relation the value of $\partial_t^2u_\nu$, given by (\ref{(N)eq_ni}).

Integrating by parts and taking advantage of the spectral localisation of $u_\nu$, we have
\begin{eqnarray*}
 & & \int\frac{2}{\sqrt{a_\veps}}\,\,\Re\left(\partial_x(a\partial_xu_\nu)\,\cdot\,
\overline{\left(\partial_tu_\nu+\frac{\partial_t\sqrt{a_\veps}}{2\sqrt{a_\veps}}u_\nu\right)}\right)\,dx\,= \\
 & & \qquad\qquad\qquad\qquad\qquad\qquad\qquad\qquad=\,\int2\frac{\partial_x\sqrt{a_\veps}}{a_\veps}\,a\,\,\Re\left(\partial_xu_\nu\,\cdot\,
\overline{\left(\partial_tu_\nu+\frac{\partial_t\sqrt{a_\veps}}{2\sqrt{a_\veps}}u_\nu\right)}\right)\,dx \\
 & & \qquad\qquad\qquad\qquad\qquad\qquad\qquad\qquad-\,\int\frac{\partial_t\sqrt{a_\veps}}{a_\veps}\,a\,|\partial_xu_\nu|^2\,dx \\
 & & \qquad\qquad\qquad\qquad\qquad\qquad\qquad\qquad-\,\int2\frac{a}{\sqrt{a_\veps}}\,\,\Re\left(\partial_xu_\nu\,\cdot\,
\overline{\partial_x\partial_tu_\nu}\right)\,dx \\
 & & \qquad\qquad\qquad\qquad\qquad\qquad\qquad\qquad-\,\int\frac{a}{\sqrt{a_\veps}}\,\,
\partial_x\left(\frac{\partial_t\sqrt{a_\veps}}{\sqrt{a_\veps}}\right)\,
\Re(\partial_xu_\nu\,\cdot\,\overline{u_\nu})\,dx\,;
\end{eqnarray*}
taken care of the fact that $b_0$, $b_1$ and $c$ are real-valued, finally we obtain the complete expression for the time derivative of the approximate
energy:
\begin{eqnarray} \label{(N)eq_dt_e_nu,eps}
 & & \frac{d}{dt}e_{\nu,\veps}(t)\,=\,\int\frac{2}{\sqrt{a_\veps}}\,\,\Re
\left((Lu)_\nu\,\cdot\,
\overline{\left(\partial_tu_\nu+\frac{\partial_t\sqrt{a_\veps}}{2\sqrt{a_\veps}}u_\nu\right)}\right)\,dx \\
 & & \qquad\qquad\quad\,+\,\int\frac{2}{\sqrt{a_\veps}}\,\,\Re\left(R_\veps u_\nu\,\cdot\,
\overline{\left(\partial_tu_\nu+\frac{\partial_t\sqrt{a_\veps}}{2\sqrt{a_\veps}}u_\nu\right)}\right)\,dx \nonumber \\
 & & \qquad\qquad\quad\,+\,\int\partial_t\sqrt{a_\veps}\left(1-\frac{a}{a_\veps}\right)|\partial_xu_\nu|^2\,dx \nonumber \\
 & & \quad\qquad\qquad\,+\,\int2\left(\sqrt{a_\veps}-\frac{a}{\sqrt{a_\veps}}\right)\Re(\partial_xu_\nu\,\cdot\,
\overline{\partial_x\partial_tu_\nu})\,dx \nonumber \\
 & & \qquad\qquad\quad\,+\,\int2\frac{\partial_x\sqrt{a_\veps}}{a_\veps}\,a\,\,\Re\left(\partial_xu_\nu\,\cdot\,
\overline{\left(\partial_tu_\nu+\frac{\partial_t\sqrt{a_\veps}}{2\sqrt{a_\veps}}u_\nu\right)}\right)\,dx \nonumber \\
 & & \qquad\qquad\quad\,-\,
\int\frac{a}{\sqrt{a_\veps}}\,\,\partial_x\left(\frac{\partial_t\sqrt{a_\veps}}{\sqrt{a_\veps}}\right)\,
\Re(\partial_xu_\nu\,\cdot\,\overline{u_\nu})\,dx \nonumber \\
 & & \quad\qquad\qquad\,+\,\int2\,\Re(u_\nu\,\cdot\,\overline{\partial_tu_\nu})\,dx \nonumber \\
 & & \quad\qquad\qquad\,+\,\int\frac{2}{\sqrt{a_\veps}}\,\,\Re\left((\partial_x([\vphi_\nu(D_x),a]\partial_xu))\,\cdot\,
\overline{\left(\partial_tu_\nu+\frac{\partial_t\sqrt{a_\veps}}{2\sqrt{a_\veps}}u_\nu\right)}\right)\,dx \nonumber \\
 & & \qquad\qquad\quad\,-\,\int\frac{2}{\sqrt{a_\veps}}\,b_0(t,x)\,\,\Re\left(\partial_tu_\nu\,\cdot\,
\overline{\left(\partial_tu_\nu+\frac{\partial_t\sqrt{a_\veps}}{2\sqrt{a_\veps}}u_\nu\right)}\right)\,dx \nonumber \\
 & & \qquad\qquad\quad\,-\,\int\frac{2}{\sqrt{a_\veps}}\,\,\Re\left(([\vphi_\nu(D_x),b_0]\partial_tu)\,\cdot\,
\overline{\left(\partial_tu_\nu+\frac{\partial_t\sqrt{a_\veps}}{2\sqrt{a_\veps}}u_\nu\right)}\right)\,dx \nonumber \\
 & & \qquad\qquad\quad\,-\,\int\frac{2}{\sqrt{a_\veps}}\,b_1(t,x)\,\,\Re\left(\partial_xu_\nu\,\cdot\,
\overline{\left(\partial_tu_\nu+\frac{\partial_t\sqrt{a_\veps}}{2\sqrt{a_\veps}}u_\nu\right)}\right)\,dx \nonumber \\
 & & \qquad\qquad\quad\,-\,\int\frac{2}{\sqrt{a_\veps}}\,\,\Re\left(([\vphi_\nu(D_x),b_1]\partial_xu)\,\cdot\,
\overline{\left(\partial_tu_\nu+\frac{\partial_t\sqrt{a_\veps}}{2\sqrt{a_\veps}}u_\nu\right)}\right)\,dx \nonumber \\
 & & \qquad\qquad\quad\,-\,\int\frac{2}{\sqrt{a_\veps}}\,c(t,x)\,\,\Re\left(u_\nu\,\cdot\,
\overline{\left(\partial_tu_\nu+\frac{\partial_t\sqrt{a_\veps}}{2\sqrt{a_\veps}}u_\nu\right)}\right)\,dx \nonumber \\
 & & \qquad\qquad\quad\,-\,\int\frac{2}{\sqrt{a_\veps}}\,\,\Re\left(([\vphi_\nu(D_x),c]u)\,\cdot\,
\overline{\left(\partial_tu_\nu+\frac{\partial_t\sqrt{a_\veps}}{2\sqrt{a_\veps}}u_\nu\right)}\right)\,dx \nonumber\;.
\end{eqnarray}

\subsection{Estimate for the approximate energy}
We want to obtain an energy estimate; so, let us start to control each term of (\ref{(N)eq_dt_e_nu,eps}).\\
Through the rest of the proof, we will denote with $C$, $C'$, $C''$ and $\what{C}$ constants depending only on $\lambda_0$, $\Lambda_0$, $C_0$ and on
the norms of the coefficients of the lower order terms of the operator $L$ in their respective functional spaces, and which are allowed to vary
from line to line.

\subsubsection{Terms with $a$ and $a_\veps$}

Thanks to relations (\ref{(N)eq_condiz_1.1}), (\ref{(N)eq_d_t_a_eps}) with $\sigma=1/2$, (\ref{(N)eq_d_t^2_a_eps}) and Bernstein's inequalities,
we deduce that there exists $C>0$, depending only on $\lambda_0$, $\Lambda_0$ and $C_0$, such that, for all $\nu\in\mbb{N}$,
$$
\left|\int\frac{2}{\sqrt{a_\veps}}\,\,\Re\left(R_\veps u_\nu\,\cdot\,
\overline{\left(\partial_tu_\nu+\frac{\partial_t\sqrt{a_\veps}}{2\sqrt{a_\veps}}u_\nu\right)}
\right)\,dx\right|\,\leq\,C\,\frac{1}{\veps}\,\log\left(\frac{1}{\veps}+1\right)\,
2^{-\nu}\,e_{\nu,\veps}(t)\;.
$$

In the same way, from (\ref{(N)eq_condiz_1.1}), (\ref{(N)eq_a_eps-a}) and (\ref{(N)eq_d_t_a_eps}), we have
$$
\left|\int\partial_t\sqrt{a_\veps}\left(1-\frac{a}{a_\veps}\right)|\partial_xu_\nu|^2\,dx\right|\,\leq\,
C\,\log\left(\frac{1}{\veps}+1\right)\,e_{\nu,\veps}(t)\;,
$$
for a constant $C$ depending again only on $\lambda_0$, $\Lambda_0$ and $C_0$.

Moreover, again from (\ref{(N)eq_condiz_1.1}) and (\ref{(N)eq_a_eps-a}) and Bernstein's inequalities, we  obtain
\begin{eqnarray*}
\left|\int2\left(\sqrt{a_\veps}-\frac{a}{\sqrt{a_\veps}}\right)\Re(\partial_xu_\nu\,\cdot\,
\overline{\partial_x\partial_tu_\nu})\,dx\right| & \leq & C\,\veps\,\log\left(\frac{1}{\veps}+1\right)\,
\|\partial_xu_\nu\|_{L^2}\,\|\partial_x\partial_tu_\nu\|_{L^2} \\
 & \leq & C\,\veps\,\log\left(\frac{1}{\veps}+1\right)\,2^{\nu+1}\,
\|\partial_xu_\nu\|_{L^2}\,\|\partial_tu_\nu\|_{L^2}\;;
\end{eqnarray*}
but we have
$$
\|\partial_tu_\nu\|_{L^2}\,\leq\,\left\|\partial_tu_\nu+\frac{\partial_t\sqrt{a_\veps}}{2\sqrt{a_\veps}}u_\nu
\right\|_{L^2}\,+\,\left\|\frac{\partial_t\sqrt{a_\veps}}{2\sqrt{a_\veps}}u_\nu\right\|_{L^2}
$$
and
\begin{eqnarray*}
\left\|\frac{\partial_t\sqrt{a_\veps}}{2\sqrt{a_\veps}}u_0\right\|_{L^2} & \leq & C\,\veps^{-1/2}\|u_0\|_{L^2} \\
\left\|\frac{\partial_t\sqrt{a_\veps}}{2\sqrt{a_\veps}}u_\nu\right\|_{L^2} & \leq & C\,\veps^{-1/2}\,2^{-\nu}\|u_\nu\|_{L^2}\qquad(\nu\geq1)\;,
\end{eqnarray*}
that give us the following:
$$
\left|\int2\left(\sqrt{a_\veps}-\frac{a}{\sqrt{a_\veps}}\right)\Re(\partial_xu_\nu\cdot
\overline{\partial_x\partial_tu_\nu})dx\right|\,\leq\,
C(\veps\,2^\nu+1)\log\left(\frac{1}{\veps}+1\right)e_{\nu,\veps}(t)\;.
$$

In a very similar way, from (\ref{(N)eq_d_x_a_eps}) one has
$$
\left|\int2\frac{\partial_x\sqrt{a_\veps}}{a_\veps}\,a\,\,\Re\left(\partial_xu_\nu\cdot
\overline{\left(\partial_tu_\nu+\frac{\partial_t\sqrt{a_\veps}}{2\sqrt{a_\veps}}u_\nu\right)}\right)dx\right|
\,\leq\,C\log\left(\frac{1}{\veps}+1\right)e_{\nu,\veps}(t)\;;
$$
moreover, from (\ref{(N)eq_d_t_a_eps}) with $\sigma=1/2$, (\ref{(N)eq_d_x_a_eps}) and (\ref{(N)eq_d_t_d_x_a_eps}) we deduce
$$
\left|\int\frac{a}{\sqrt{a_\veps}}\,\,\partial_x\left(\frac{\partial_t\sqrt{a_\veps}}{\sqrt{a_\veps}}\right)\,
\Re(\partial_xu_\nu\,\cdot\,\overline{u_\nu})\,dx\right|\,\leq\,C\,\frac{1}{\veps}\,
\log\left(\frac{1}{\veps}+1\right)\,2^{-\nu}\,e_{\nu,\veps}(t)\;.
$$

Finally, we have
$$
\left|\int2\,\Re(u_\nu\,\cdot\,\overline{\partial_tu_\nu})\,dx\right|\,\leq\,
C\,\veps^{-1/2}\,2^{-\nu}\,e_{\nu,\veps}(t)\;.
$$


\subsubsection{Terms with $b_0$, $b_1$ and $c$}
Thanks to the hypothesis (\ref{(N)eq_condiz_B})-(\ref{(N)eq_condiz_c}), one has that there exist suitable constants, depending only on $\lambda_0$, $\Lambda_0$, $C_0$
and on the norms of $b_0$ and $b_1$ in the space $L^\infty(\mbb{R}_t;\mc{C}^{\omega}(\mbb{R}_x))$ and of $c$ in
$L^\infty(\mbb{R}_t\times\mbb{R}_x)$, such that
\begin{eqnarray*}
& & \left|\int\frac{2}{\sqrt{a_\veps}}\,b_0(t,x)\,\,\Re\left(\partial_tu_\nu\,\cdot\,
\overline{\left(\partial_tu_\nu+\frac{\partial_t\sqrt{a_\veps}}{2\sqrt{a_\veps}}u_\nu\right)}\right)\,dx\right|\,\leq \\
& & \qquad\qquad\qquad\qquad\qquad\qquad\qquad\leq\,C\,\|\partial_t u_\nu\|_{L^2}\,
\left\|\partial_tu_\nu+\frac{\partial_t\sqrt{a_\veps}}{2\sqrt{a_\veps}}u_\nu\right\|_{L^2} \\
& & \qquad\qquad\qquad\qquad\qquad\qquad\qquad\leq\,(C+C'\,2^{-\nu}\veps^{-1/2})\,e_{\nu,\veps}(t)\;; \\
& & \left|\int\frac{2}{\sqrt{a_\veps}}\,b_1(t,x)\,\,\Re\left(\partial_xu_\nu\,\cdot\,
\overline{\left(\partial_tu_\nu+\frac{\partial_t\sqrt{a_\veps}}{2\sqrt{a_\veps}}u_\nu\right)}\right)\,dx\right|
\,\leq \\
& & \qquad\qquad\qquad\qquad\qquad\qquad\qquad\leq\,C\,\int\sqrt[4]{a_{\veps}}|\partial_x u_\nu|\,
\frac{1}{\sqrt[4]{a_\veps}}\left|\partial_tu_\nu+
\frac{\partial_t\sqrt{a_{_\veps}}}{2\sqrt{a_\veps}}u_\nu\right|\,dx \\
& & \qquad\qquad\qquad\qquad\qquad\qquad\qquad\leq\,2C\,\int\sqrt{a_\veps}|\partial_x u_\nu|^2\,+\,
\frac{1}{\sqrt{a_\veps}}\left|\partial_tu_\nu+\frac{\partial_t\sqrt{a_\veps}}{2\sqrt{a_\veps}}u_\nu
\right|^2\,dx \\
& & \qquad\qquad\qquad\qquad\qquad\qquad\qquad\leq\,2\,e_{\nu,\veps}(t)\;; \\
& & \left|\int\frac{2}{\sqrt{a_\veps}}\,c(t,x)\,\,\Re\left(u_\nu\,\cdot\,
\overline{\left(\partial_tu_\nu+\frac{\partial_t\sqrt{a_\veps}}{2\sqrt{a_\veps}}u_\nu\right)}\right)\,dx\right|\,\leq \\
& & \qquad\qquad\qquad\qquad\qquad\qquad\qquad\leq\,C\,\int|u_\nu|\,\frac{1}{\sqrt[4]{a_\veps}}
\left|\partial_tu_\nu+\frac{\partial_t\sqrt{a_\veps}}{2\sqrt{a_\veps}}u_\nu\right|\,dx \\
& & \qquad\qquad\qquad\qquad\qquad\qquad\qquad\leq\,2C\,e_{\nu,\veps}(t)\;,
\end{eqnarray*}
where we have delt with $\|\partial_tu_\nu\|_{L^2}$ as before.

\bigskip
Now, we join the approximation parameter $\veps$ with the dual variable $\xi$, setting
$$
\veps\,=\,2^{-\nu}\;;
$$
so, from (\ref{(N)eq_dt_e_nu,eps}) and the previous inequalities, we obtain
\begin{eqnarray}
& & \frac{d}{dt}e_{\nu,2^{-\nu}}(t)\,\leq\,\wtilde{C}\,(\nu+1)\,e_{\nu,2^{-\nu}} \label{(N)eq_dt_e_nu,eps_II}\\
& & \qquad\qquad\quad\quad\,+\,\int\frac{2}{\sqrt{a_{2^{-\nu}}}}\,\,\Re\left((Lu)_\nu\,\cdot\,
\overline{\left(\partial_tu_\nu+\frac{\partial_t\sqrt{a_{2^{-\nu}}}}{2\sqrt{a_{2^{-\nu}}}}
u_\nu\right)}\right)\,dx \nonumber \\
& & \qquad\qquad\quad\quad\,+\,\int\frac{2}{\sqrt{a_{2^{-\nu}}}}\,\,\Re\left((\partial_x([\vphi_\nu(D_x),a]\partial_xu))\,\cdot\,
\overline{\left(\partial_tu_\nu+\frac{\partial_t\sqrt{a_{2^{-\nu}}}}{2\sqrt{a_{2^{-\nu}}}}
u_\nu\right)}\right)\,dx \nonumber \\
& & \qquad\qquad\quad\quad\,-\,\int\frac{2}{\sqrt{a_{2^{-\nu}}}}\,\,\Re\left(([\vphi_\nu(D_x),b_0]\partial_tu)\,\cdot\,
\overline{\left(\partial_tu_\nu+\frac{\partial_t\sqrt{a_{2^{-\nu}}}}{2\sqrt{a_{2^{-\nu}}}}
u_\nu\right)}\right)\,dx \nonumber \\
& & \qquad\qquad\quad\quad\,-\,\int\frac{2}{\sqrt{a_{2^{-\nu}}}}\,\,\Re\left(([\vphi_\nu(D_x),b_1]\partial_xu)\,\cdot\,
\overline{\left(\partial_tu_\nu+\frac{\partial_t\sqrt{a_{2^{-\nu}}}}{2\sqrt{a_{2^{-\nu}}}}
u_\nu\right)}\right)\,dx \nonumber \\
& & \qquad\qquad\quad\quad\,-\,\int\frac{2}{\sqrt{a_{2^{-\nu}}}}\,\,\Re\left(([\vphi_\nu(D_x),c]u)\,\cdot\,
\overline{\left(\partial_tu_\nu+\frac{\partial_t\sqrt{a_{2^{-\nu}}}}{2\sqrt{a_{2^{-\nu}}}}
u_\nu\right)}\right)\,dx \nonumber\;,
\end{eqnarray}
for a suitable constant $\wtilde{C}$, which depends only on
$\lambda_0$, $\Lambda_0$, $C_0$ and on the norms of the coefficients of the operator $L$ in their respective functional spaces.

\subsection{Estimates for the commutator terms}
Now, we have to deal with the commutator terms. As we will see, it's useful to consider immediately the sum over $\nu\in\mbb{N}$.

First, we report an elementary lemma (see also~\cite{Colomb-DelSant_2}), which we will use very often in the next calculations.
\begin{lemma} \label{(N)lemma_1}
There exist two continuous, decreasing functions $\alpha_{1}$, $\alpha_2:\,]0,1[\,\rightarrow\,]0,+\infty[$ such that
$\lim_{c\rightarrow0^+}\alpha_j(c)=+\infty$ for $j=1\,,\,2$ and such that, for all $\delta\in]0,1]$ and all $n\geq1$, the following inequalities hold:
$$
\sum_{j=1}^{n}\,e^{\delta j}\,j^{-1/2}\,\leq\,\alpha_1(\delta)\,e^{\delta n}\,n^{-1/2}\,,\qquad\qquad
\sum_{j=n}^{+\infty}\,e^{-\delta j}\,j^{1/2}\,\leq\,\alpha_2(\delta)\,e^{-\delta n}\,n^{1/2}\,.
$$
\end{lemma}

Before going on, we take $\beta>0$ and $T\in]0,T_0]$ such that
$$
\beta T\,=\,\frac{\theta}{2}\,\log2\;.
$$
\begin{oss} \label{oss:beta_T}
Notice that, thanks to the hypothesis of Theorem \ref{(N)teor_teorema1}, this condition implies that
$$
\beta T\,\leq\,\frac{\omega-\theta}{2}\;.
$$
Moreover, for all $t\in[0,T]$, we have:
\begin{eqnarray*}
\beta t+\frac{\theta}{2}\log2 & \geq & \frac{\theta}{2}\log2\;>0\\
\beta t+\frac{\theta}{2}\log2 & \leq & \theta\log2\;\leq\;\frac{1}{2}\log2\;<1\\
(1-\theta)\log2-\beta t & \geq & \left(1-\frac{3}{2}\theta\right)\log2\;\geq\;\left(1-\frac{3}{4}\right)\log2\;>0\\
(1-\theta)\log2-\beta t & \leq & (1-\theta)\log2\;\leq\;\log2\;<1\,.
\end{eqnarray*}
\end{oss}

Finally, we set (with the same notations used in the subsection \ref{sub-sec:L-P_dec})
$$
\psi_{\mu}\,=\,\vphi_{\mu-1}+\vphi_{\mu}+\vphi_{\mu+1}\qquad(\vphi_{-1}\equiv0)\;.
$$
As $\psi_\mu\equiv1$ on the support of $\vphi_\mu$, we can write
$$
\partial_xu_\mu\,=\,\vphi_\mu(D_x)\partial_xu\,=\,\Psi_\mu(\vphi_\mu(D_x)\partial_xu)\,=\,\Psi_\mu\partial_xu_\mu\;,
$$
where $\Psi_\mu$ is the operator related to $\psi_\mu$. So, given a generic function $f(t,x)$, one has
\begin{equation} \label{(N)eq_nu-mu}
[\vphi_\nu(D_x),f]\partial_xu\,=\,[\vphi_\nu(D_x),f]\left(\sum_{\mu\geq0}\partial_xu_\mu\right)\,=\,
\sum_{\mu\geq0}([\vphi_\nu(D_x),f]\Psi_\mu)\partial_xu_\mu\;.
\end{equation}

After these preliminary remarks, we now can go on with commutators' estimates.

\subsubsection{Term with $[\vphi_\nu(D_x),a]$}

Due to Bernstein's inequalities, we have
$$
\left\|\partial_x\left(\partial_tu_\nu+\frac{\partial_t\sqrt{a_{2^{-\nu}}}}{2\sqrt{a_{2^{-\nu}}}}
u_\nu\right)\right\|_{L^2}\,\leq\,C\,2^\nu\,\left(e_{\nu,2^{-\nu}}(t)\right)^{1/2}\;.
$$
So, using (\ref{(N)eq_nu-mu}) and the fact that $a_\veps$ is real-valued, one has
\begin{eqnarray*}
& & \left|\int\frac{2}{\sqrt{a_{2^{-\nu}}}}\,\,\Re\left(\partial_x([\vphi_\nu(D_x),a]\partial_xu)\,\cdot\,
\overline{\left(\partial_tu_\nu+\frac{\partial_t\sqrt{a_{2^{-\nu}}}}{2\sqrt{a_{2^{-\nu}}}}
u_\nu\right)}\right)\,dx\right|\,\leq\, \\
& & \qquad\qquad\qquad\qquad\qquad\qquad\leq\,C\,\sum_\mu\|([\vphi_\nu(D_x),a]\Psi_\mu)\partial_xu_\mu\|_{L^2}\,
2^\nu\,\left(e_{\nu,2^{-\nu}}(t)\right)^{1/2} \\
& & \qquad\qquad\qquad\qquad\qquad\qquad\leq\,C\,\sum_\mu\|[\vphi_\nu(D_x),a]\Psi_\mu\|_{\mc{L}(L^2)}\,\left(e_{\mu,2^{-\mu}}(t)\right)^{1/2}\,
2^\nu\,\left(e_{\nu,2^{-\nu}}(t)\right)^{1/2}\;,
\end{eqnarray*}
with the constant $C$ which depends only on $\lambda_0$, $\Lambda_0$ and $C_0$.\\
Then,
\begin{eqnarray*}
& & \left|\sum_{\nu\geq0}e^{-2\beta(\nu+1)t}2^{-2\nu\theta}
\int\frac{2}{\sqrt{a_{2^{-\nu}}}}\,\,\Re\left(\partial_x([\vphi_\nu(D_x),a]\partial_xu)
\overline{\left(\partial_tu_\nu+\frac{\partial_t\sqrt{a_{2^{-\nu}}}}{2\sqrt{a_{2^{-\nu}}}}
u_\nu\right)}\right)dx\right|\,\leq \\
& & \;\quad\qquad\leq\,C\,\sum_{\nu,\mu}k_{\nu\mu}\,(\nu+1)^{1/2}e^{-\beta(\nu+1)t}2^{-\nu\theta}
\left(e_{\nu,2^{-\nu}}\right)^{1/2}\,(\mu+1)^{1/2}e^{-\beta(\mu+1)t}2^{-\mu\theta}\left(e_{\mu,2^{-\mu}}\right)^{1/2}\;,
\end{eqnarray*}
where we have set
\begin{equation} \label{(N)eq_kernel_comm[a]}
k_{\nu\mu}\,=\,e^{-(\nu-\mu)\beta t}\,2^{-(\nu-\mu)\theta}\,2^{\nu}\,(\nu+1)^{-1/2}\,(\mu+1)^{-1/2}\,
\|[\vphi_\nu(D_x),a]\Psi_\mu\|_{\mc{L}(L^2)}\;.
\end{equation}
Observe that, if $|\nu-\mu|\geq3$, then $\vphi_\nu\psi_\mu\equiv0$, so $[\vphi_\nu(D_x),a]\Psi_\mu=\vphi_\nu(D_x)[a,\Psi_\mu]$.
Therefore, from lemma \ref{lemma:comm}, in particular from (\ref{eq:comm_Log-Lip}), we deduce that
$$
\|[\vphi_{\nu}(D_x),a(t,x)]\Psi_\mu\|_{\mc{L}(L^2)}\leq
\left\{
\begin{array}{ll}
C\,2^{-\nu}(\nu+1) & \mbox{ if } |\nu-\mu|\leq2\;, \\
 & \\
C\,2^{-max\{\nu,\mu\}}(\max\{\nu,\mu\}+1) & \mbox{ if } |\nu-\mu|\geq3\;,
\end{array}
\right.
$$
where the constant $C$ depends only on $\Lambda_0$ and $C_0$.

Now our aim is to apply Schur's lemma, so to estimate the quantity
\begin{equation} \label{(N)eq_lemma_schur}
\sup_{\mu}\sum_{\nu}|k_{\nu\mu}|\,+\,\sup_{\nu}\sum_{\mu}|k_{\nu\mu}|\;.
\end{equation}
To do this, we will use lemma \ref{(N)lemma_1} and the inequalities stated in remark (\ref{oss:beta_T}).

\begin{enumerate}
 \item Fix $\mu\leq2$.
\begin{eqnarray*}
\sum_{\nu\geq0}|k_{\nu\mu}| & \leq & C\,e^{(\mu+1)\beta t}\,2^{(\mu+1)\theta}\,(\mu+1)^{-\frac{1}{2}}
\sum_\nu e^{-(\nu+1)\beta t}\,2^{-(\nu+1)\theta}\,(\nu+1)^{\frac{1}{2}} \\
& = & C\,e^{(\mu+1)\beta t}\,2^{(\mu+1)\theta}\,(\mu+1)^{-\frac{1}{2}}
\sum_\nu e^{-(\nu+1)(\beta t+\theta\log2)}\,(\nu+1)^{\frac{1}{2}} \\
& \leq & C\,e^{3\beta t}\,2^{3\theta}\,\alpha_2(\beta t+\theta\log2) \\
& \leq & C\,2^{\frac{9}{2}\theta}\,\alpha_2(\theta\log2)\;.
\end{eqnarray*}
 \item Now, take $\mu\geq3$ and first consider
\begin{eqnarray*}
\sum_{\nu=0}^{\mu-3}|k_{\nu\mu}| & \leq & C\,e^{(\mu+1)\beta t}\,2^{-(\mu+1)(1-\theta)}\,(\mu+1)^{\frac{1}{2}}
\sum_{\nu=0}^{\mu-3} e^{-(\nu+1)\beta t}\,2^{(\nu+1)(1-\theta)}\,(\nu+1)^{-\frac{1}{2}} \\
& \leq & C\,e^{(\mu+1)\beta t}\,2^{-(\mu+1)(1-\theta)}\,(\mu+1)^{\frac{1}{2}}
\sum_{\nu=0}^{\mu-3}e^{(\nu+1)(-\beta t+(1-\theta)\log2)}\,(\nu+1)^{-\frac{1}{2}} \\
& \leq & C\,e^{(\mu+1)\beta t}\,2^{-(\mu+1)(1-\theta)}\,(\mu+1)^{\frac{1}{2}}\,
\alpha_1(-\beta t+(1-\theta)\log2)\,\cdot \\
& & \qquad\qquad\qquad\qquad\qquad\qquad\qquad\quad\;\cdot\,
e^{(-\beta t+(1-\theta)\log2)(\mu-2)}\,(\mu-2)^{-\frac{1}{2}} \\
& \leq & C\,2^{\frac{9}{2}\theta}\,\alpha_1\left(\left(1-\frac{3}{2}\theta\right)\log2\right)\;.
\end{eqnarray*}

For the second part of the sum, one has
\begin{eqnarray*}
\sum_{\nu=\mu-2}^{+\infty}|k_{\nu\mu}| & \leq & C\,e^{(\mu+1)\beta t}\,2^{(\mu+1)\theta}\,(\mu+1)^{-1/2}
\sum_{\nu=\mu-2}^{+\infty}e^{-(\nu+1)\beta t}\,2^{-(\nu+1)\theta}\,(\nu+1)^{1/2} \\
& \leq & C\,e^{(\mu+1)\beta t}\,2^{(\mu+1)\theta}\,(\mu+1)^{-1/2}\alpha_2(\beta t+\theta\log2)\,\cdot \\
& & \qquad\qquad\qquad\qquad\qquad\qquad\qquad\qquad\cdot\,e^{-(\beta t+\theta\log2)(\mu-1)}\,(\mu-1)^{1/2} \\
& \leq & C\,2^{\frac{7}{2}\theta}\,\alpha_2(\theta\log2)\;.
\end{eqnarray*}
 \item Fix now $\nu\geq0$; we have
\begin{eqnarray*}
\sum_{\mu=0}^{\nu+2}|k_{\nu\mu}| & \leq & C\,e^{-(\nu+1)\beta t}\,2^{-(\nu+1)\theta}\,(\nu+1)^{1/2}
\sum_{\mu=0}^{\nu+2}e^{(\mu+1)\beta t}\,2^{(\mu+1)\theta}\,(\mu+1)^{-1/2} \\
& \leq & C\,e^{-(\nu+1)\beta t}\,2^{-(\nu+1)\theta}\,(\nu+1)^{1/2}\,\alpha_1(\beta t+\theta\log2)\,\cdot \\
& & \qquad\qquad\qquad\qquad\qquad\qquad\qquad\qquad\cdot\,e^{(\beta t+\theta\log2)(\nu+3)}\,(\nu+3)^{-1/2} \\
& \leq & C\,2^{\frac{7}{2}\theta}\,\alpha_1(\theta\log2)\;.
\end{eqnarray*}
For the second part of the series, the following inequality holds:
\begin{eqnarray*}
\sum_{\mu=\nu+3}^{+\infty}|k_{\nu\mu}| & \leq & C\,e^{-(\nu+1)\beta t}\,2^{(\nu+1)(1-\theta)}\,(\nu+1)^{-\frac{1}{2}}
\sum_{\mu=\nu+3}^{+\infty}e^{(\mu+1)\beta t}\,2^{-(\mu+1)(1-\theta)}\,(\mu+1)^{\frac{1}{2}} \\
& \leq & C\,e^{-(\nu+1)\beta t}\,2^{(\nu+1)(1-\theta)}\,(\nu+1)^{-\frac{1}{2}}\,
\alpha_2(-\beta t+(1-\theta)\log2)\,\cdot \\
& & \qquad\qquad\qquad\qquad\qquad\qquad\qquad\qquad\;\;\cdot\,
e^{(-\beta t+(1-\theta)\log2)(\nu+4)}\,(\nu+4)^{\frac{1}{2}} \\
& \leq & C\,2^{\frac{9}{2}\theta}\,\alpha_2\left(\left(1-\frac{3}{2}\theta\right)\log2\right)\;.
\end{eqnarray*}
\end{enumerate}

In conclusion, there exists a positive function $\varPi$, with $\lim_{\theta\rightarrow0^+}\varPi(\theta)=+\infty$, such that
$$
\sup_{\mu}\sum_{\nu}|k_{\nu\mu}|\,+\,\sup_{\nu}\sum_{\mu}|k_{\nu\mu}|\,\leq\,C\,\varPi(\theta)\;,
$$
and so
\begin{eqnarray*}
& & \hspace{-1cm}\left|\sum_{\nu\geq0}e^{-2\beta(\nu+1)t}\,2^{-2\nu\theta}
\int\frac{2}{\sqrt{a_{2^{-\nu}}}}\,\,\Re\left(\partial_x([\vphi_\nu(D_x),a]\partial_xu)
\overline{\left(\partial_tu_\nu+\frac{\partial_t\sqrt{a_{2^{-\nu}}}}{2\sqrt{a_{2^{-\nu}}}}
u_\nu\right)}\right)dx\right|\,\leq \\
& & \qquad\qquad\qquad\qquad\qquad\qquad\quad\qquad\qquad \leq\,C\,\varPi(\theta)\,
\sum_{\nu=0}^{+\infty}(\nu+1)\,e^{-2\beta(\nu+1)t}\,2^{-2\nu\theta}\,e_{\nu,2^{-\nu}}(t)\;.
\end{eqnarray*}

\subsubsection{Terms with $[\vphi_\nu(D_x),b_0]$ and $[\vphi_\nu(D_x),b_1]$}

Now, let us consider
\begin{eqnarray*}
& & \left|\int\frac{2}{\sqrt{a_{2^{-\nu}}}}\,\,\Re\left([\vphi_{\nu}(D_x),b_0(t,x)]\partial_t u\,
\overline{\left(\partial_tu_\nu+\frac{\partial_t\sqrt{a_{2^{-\nu}}}}{2\sqrt{a_{2^{-\nu}}}}
u_\nu\right)}\right)\,dx\right|\,\leq \\
& & \quad\qquad\qquad\qquad\qquad\qquad\qquad\leq\,2\,\|[\vphi_{\nu}(D_x),b_0(t,x)]\partial_t u\|_{L^2}\,
\left\|\frac{1}{\sqrt{a_{2^{-\nu}}}}\left|\partial_tu_\nu+\frac{\partial_t\sqrt{a_{2^{-\nu}}}}{2\sqrt{a_{2^{-\nu}}}}
u_\nu\right|\;\right\|_{L^2} \\
& & \quad\qquad\qquad\qquad\qquad\qquad\qquad\leq\,2\,\|[\vphi_{\nu}(D_x),b_0(t,x)]\partial_t u\|_{L^2}\,(e_{\nu,2^{-\nu}}(t))^{1/2}\;.
\end{eqnarray*}

Thanks to relation (\ref{(N)eq_nu-mu}), we have
\begin{eqnarray*}
\|[\vphi_{\nu}(D_x),b_0(t,x)]\partial_t u\|_{L^2} & = & 
\left\|[\vphi_{\nu}(D_x),b_0(t,x)]\sum_{\mu\geq0}\Psi_\mu\partial_tu_\mu\right\|_{L^2} \\
 & \leq & \sum_{\mu\geq0}\|[\vphi_{\nu}(D_x),b_0(t,x)]\Psi_\mu\|_{\mc{L}(L^2)}\|\partial_tu_\mu\|_{L^2}\;.
\end{eqnarray*}
As we have done before, we have, for constants depending only on $\lambda_0$, $\Lambda_0$ and $C_0$,
\begin{eqnarray*}
\|\partial_tu_\mu\|_{L^2}\,=\,
\left(\int\frac{\sqrt{a_{2^{-\mu}}}}{\sqrt{a_{2^{-\mu}}}}|\partial_tu_\mu|^{2}dx\right)^{1/2} & \leq & \left(C\,e_{\mu,2^{-\mu}}(t)+
C'\,\left\|\frac{\partial_t\sqrt{a_{2^{-\mu}}}}{2\sqrt{a_{2^{-\mu}}}}u_\mu\right\|_{L^2}\right)^{1/2} \\
& \leq & \left(C e_{\mu,2^{-\mu}}(t)\,+\,C'\frac{1}{2^{-\mu}}2^{-2\mu}\|\partial_x u_\mu\|_{L^2}\right)^{1/2} \\
& \leq & (C+ C'\,2^{-\mu})^{1/2}\left(e_{\mu,2^{-\mu}}\right)^{1/2} \\
& \leq & C\,\left(e_{\mu,2^{-\mu}}\right)^{1/2}\;.
\end{eqnarray*}

Due to lemma \ref{lemma:comm}, we obtain
$$
\|[\vphi_{\nu}(D_x),b_0(t,x)]\Psi_\mu\|_{\mc{L}(L^2)}\leq
\left\{
\begin{array}{ll}
C\,2^{-\nu\omega} & \quad\mbox{ if } |\nu-\mu|\leq2 \\
 & \\
C\,2^{-max\{\mu,\nu\}\omega} & \quad\mbox{ if } |\nu-\mu|\geq3
\end{array}
\right.
$$
where $C$ is a constant depending only on $\|b_0\|_{L^\infty(\mbb{R}_t;\,\mc{C}^{\omega}(\mbb{R}_x))}$.\\
As a matter of fact, the kernel of operator $[\vphi_\nu(D_x),b_0]$ is
$$
n(x,y)\,=\,\what{\vphi}(2^\nu(x-y))\,2^{\nu}\,(b_0(t,x)-b_0(t,y))\;;
$$
so, to evaluate its norm we apply Schur's lemma and, thanks to the fact that $b_0$ is $\omega$-h\"older, we get the desired estimate.

Therefore,
\begin{eqnarray*}
& & \left|\sum_{\nu\geq0}e^{-2\beta t(\nu+1)}2^{-2\nu\theta} \int\frac{2}{\sqrt{a_{2^{-\nu}}}}\,\,\Re\left([\Delta_{\nu},b_0]\partial_tu\,
\overline{\left(\partial_t u_\nu+
\frac{\partial_t\sqrt{a_{2^{-\nu}}}}{2\sqrt{a_{2^{-\nu}}}}u_\nu\right)}\right)dx\right|\,\leq \\
& & \qquad\qquad\qquad\qquad\qquad\qquad\leq\,\sum_{\nu,\mu\geq0}e^{-\beta t(\nu+1)}\,2^{-\nu\theta}\,\left(e_{\nu,2^{-\nu}}\right)^{1/2}\,
e^{-\beta t(\mu+1)}\,2^{-\mu\theta}\,\left(e_{\mu,2^{-\mu}}\right)^{1/2}\,l_{\nu\mu}\;,
\end{eqnarray*}
where we have defined
\begin{equation} \label{(N)eq_kernel_comm[b_0]}
l_{\nu\mu}\,:=\,e^{-(\nu-\mu)\beta t}\,2^{-(\nu-\mu)\theta}\,\|[\vphi_{\nu}(D_x),b_0(t,x)]\Psi_\mu\|_{\mc{L}(L^2)}\;.
\end{equation}

As made before, we are going to estimate $l_{\nu\mu}$ applying Schur's lemma.

\begin{enumerate}
 \item Let us fix $\mu\leq2$. Then
\begin{eqnarray*}
\sum_{\nu\geq0}|l_{\nu\mu}| & \leq & C\,e^{(\mu+1)\beta t}\,2^{(\mu+1)\theta}\,
\sum_{\nu\geq0}e^{-(\nu+1)\beta t}\,2^{-(\nu+1)\theta}\,2^{-\nu\omega} \\
& \leq & C\,e^{3\beta t}\,2^{3\theta}\,\sum_{\nu\geq0}e^{-(\nu+1)(\beta t+\theta\log2)}\,(\nu+1)^{1/2} \\
& \leq & C\,e^{3\beta t}\,2^{3\theta}\,\alpha_2(\beta t+\theta\log2) \\
& \leq & C\,2^{\frac{9}{2}\theta}\,\alpha_2(\theta\log2)\,.
\end{eqnarray*}
 \item Now, take $\mu\geq3$ and consider first
\begin{eqnarray*}
\sum_{\nu=0}^{\mu-3}|l_{\nu\mu}| & \leq & C\,e^{\mu\beta t}\,2^{\mu\theta}\,2^{-\mu\omega}
\sum_{\nu=0}^{\mu-3} e^{-\nu\beta t}\,2^{-\nu\theta} \\
& \leq & C\,e^{\mu(\beta t-(\omega-\theta)\log2)}\,(\mu-2) \\
& \leq & C\,e^{-\mu(\omega-\frac{\theta}{2})\log2}\,(\mu-2) \\
& \leq & C\,M(\omega,\theta)\;,
\end{eqnarray*}
where $M(\omega,\theta)$ is the maximum of the function $z\mapsto e^{-\gamma z}\,(z-2)$, with $\gamma=\left(\omega-\frac{\theta}{2}\right)\log2$.\\
For the second part of the sum, we have
\begin{eqnarray*}
\sum_{\nu=\mu-2}^{+\infty}|l_{\nu\mu}| & \leq & C\,e^{(\mu+1)\beta t}\,2^{(\mu+1)\theta}
\sum_{\nu=\mu-2}^{+\infty}e^{-(\nu+1)\beta t}\,2^{-(\nu+1)\theta}\,2^{-\nu\omega}
\frac{(\nu+1)^{1/2}}{(\nu+1)^{1/2}} \\
& \leq & C\,e^{(\mu+1)\beta t}\,2^{(\mu+1)\theta}\,(\mu-1)^{-1/2}\,\alpha_2(\beta t+\theta\log2)\,\cdot \\
& & \qquad\qquad\qquad\qquad\qquad\qquad\qquad\qquad\cdot\,e^{-(\mu-1)\beta t}\,2^{-(\mu-1)\theta}\,(\mu-1)^{1/2} \\
& \leq &  C\,e^{2(\beta t+\theta\log2)}\,\alpha_2(\theta\log2) \\
& \leq & C\,2^{\frac{9}{2}\theta}\,\alpha_2(\theta\log2)\,.
\end{eqnarray*}
\item Fix now $\nu$. Initially, we have
\begin{eqnarray*}
\sum_{\mu=0}^{\nu+2}|l_{\nu\mu}| & \leq & C\,e^{-(\nu+1)\beta t}\,2^{-(\nu+1)\theta}\,2^{-\nu\omega}
\sum_{\mu=0}^{\nu+2}e^{(\mu+1)(\beta t+\theta\log2)}\,\frac{(\mu+1)^{1/2}}{(\mu+1)^{1/2}} \\
& \leq & C\,e^{-(\nu+1)\beta t}\,2^{-(\nu+1)\theta}\,2^{-\nu\omega}\,(\nu+3)^{1/2}\,\alpha_1(\beta t+\theta\log2)\,\cdot \\
& & \qquad\qquad\qquad\qquad\qquad\qquad\qquad\qquad\cdot\,e^{(\nu+3)\beta t}\,2^{(\nu+3)\theta}\,(\nu+3)^{-1/2} \\
& \leq & C\,e^{2(\beta t+\theta\log2)}\,\alpha_1(\theta\log2) \\
& \leq & C\,2^{3\theta}\,\alpha_1(\theta\log2)\;.
\end{eqnarray*}
Moreover,
\begin{eqnarray*}
\sum_{\mu=\nu+3}^{+\infty}|l_{\nu\mu}| & \leq & C\,e^{-\nu\beta t}\,2^{-\nu\theta}
\sum_{\mu=\nu+3}^{+\infty}e^{-\mu(-\beta t+(\omega-\theta)\log2)}\,\frac{\mu^{1/2}}{\mu^{1/2}} \\
& \leq & C\,e^{-\nu\beta t}\,2^{-\nu\theta}\,(\nu+3)^{-1/2}\,\alpha_2((\omega-\theta)\log2-\beta t)\,\cdot \\
& & \qquad\qquad\qquad\qquad\qquad\qquad\qquad\cdot\,e^{(\nu+3)\beta t}\,2^{-(\nu+3)(\omega-\theta)}\,(\nu+3)^{1/2}\\
& \leq & C\,e^{3\beta t}\,2^{3\theta}\,2^{-(\nu+3)\omega}\,\alpha_2\left(\left(\omega-\frac{3}{2}\theta\right)\log2\right) \\
& \leq & C\,2^{\frac{9}{2}\theta}\,\alpha_2\left(\left(\omega-\frac{3}{2}\theta\right)\log2\right)\,.
\end{eqnarray*}
\end{enumerate}

From all these inequalities, thanks to Schur's lemma, one has that there exists a constant $\wtilde{M}(\omega,\theta)$,
depending only on $\|b_0\|_{L^\infty(\mbb{R}_t;\,\mc{C}^{\omega}(\mbb{R}_x))}$ and on the parameter $\theta$ (which we have fixed at the beginning of
the calculations), such that
$$
\sup_{\mu}\sum_{\nu}|l_{\nu\mu}|\,+\,\sup_{\nu}\sum_{\mu}|l_{\nu\mu}|\,\leq\,C\,\wtilde{M}(\omega,\theta)\;;
$$
from this relation, finally we get
\begin{eqnarray*}
& & \hspace{-0.5cm}\left|\sum_{\nu\geq0}e^{-2\beta t(\nu+1)}2^{-2\nu\theta} \int\frac{2}{\sqrt{a_{2^{-\nu}}}}\,\,
\Re\left([\vphi_{\nu}(D_x),b_0]\partial_t u\,
\overline{\left(\partial_t u_\nu+\frac{\partial_t\sqrt{a_{2^{-\nu}}}}{2\sqrt{a_{2^{-\nu}}}}u_\nu\right)}\right)dx\right|\,\leq \\
& & \qquad\qquad\qquad\qquad\qquad\qquad\qquad\qquad\qquad\leq\,C\,\wtilde{M}(\omega,\theta)
\sum_{\nu=0}^{+\infty}(\nu+1)\,e^{-2\beta(\nu+1)t}\,2^{-2\nu\theta}\,e_{\nu,2^{-\nu}}(t)\;.
\end{eqnarray*}

With regard to the term with the commutator $[\vphi_\nu(D_x),b_1(t,x)]$, notice that, as we made before, one has, for constants depending
only on $\lambda_0$, $\Lambda_0$ and $C_0$,
\begin{eqnarray*}
& & \left|\int\frac{2}{\sqrt{a_{2^{-\nu}}}}\,\,\Re\left([\vphi_{\nu}(D_x),b_1]\partial_xu\,
\overline{\left(\partial_tu_\nu+\frac{\partial_t\sqrt{a_{2^{-\nu}}}}{2\sqrt{a_{2^{-\nu}}}}u_\nu\right)}
\right)\,dx\right|\,\leq \\
& & \qquad\qquad\leq\,C\,\sum_{\mu\geq0}\|[\vphi_\nu(D_x),b_1]\Psi_\mu\|_{\mc{L}(L^2)}\,
\|\sqrt[4]{a_{2^{-\nu}}}\partial_xu_\mu\|_{L^2}\,\left\|\frac{1}{\sqrt[4]{a_{2^{-\nu}}}}
\left(\partial_tu_\nu+\frac{\partial_t\sqrt{a_{2^{-\nu}}}}{2\sqrt{a_{2^{-\nu}}}}u_\nu\right)\right\|_{L^2} \\
& & \qquad\qquad\leq\,C\,\sum_{\mu\geq0}\|[\vphi_\nu(D_x),b_1]\Psi_\mu\|_{\mc{L}(L^2)}\,\left(e_{\nu,2^{-\nu}}\right)^{1/2}\left(e_{\mu,2^{-\mu}}\right)^{1/2}\;.
\end{eqnarray*}

Therefore,
\begin{eqnarray*}
& & \left|\sum_{\nu\geq0}e^{-2\beta t(\nu+1)}\,2^{-2\nu\theta}\,\int\frac{2}{\sqrt{a_{2^{-\nu}}}}\,\,\Re\left([\vphi_{\nu}(D_x),b_1]\partial_xu\,
\overline{\left(\partial_tu_\nu+\frac{\partial_t\sqrt{a_{2^{-\nu}}}}{2\sqrt{a_{2^{-\nu}}}}u_\nu\right)}
\right)dx\right|\,\leq\\
& & \qquad\qquad\qquad\qquad\qquad\leq\,C\,\sum_{\nu,\mu\geq0}e^{-\beta(\nu+1)t}\,2^{-\nu\theta}\,\left(e_{\nu,2^{-\nu}}\right)^{1/2}\,
e^{-\beta(\mu+1)t}\,2^{-\mu\theta}\,\left(e_{\mu,2^{-\mu}}\right)^{1/2}\,h_{\nu\mu}\;,
\end{eqnarray*}
where we have set again
$$
h_{\nu\mu}\,=\,e^{-(\nu-\mu)\beta t}\,2^{-(\nu-\mu)\theta}\,\|[\vphi_\nu(D_x),b_1]\Psi_\mu\|_{\mc{L}(L^2)}\;.
$$
As $b_0$ and $b_1$ satisfy to the same hypothesis, the commutator $[\vphi_\nu(D_x),b_1]$ verifies the same inequalities as $[\vphi_\nu(D_x),b_0]$;
so, if we repeat the same calculations, we obtain
\begin{eqnarray*}
& & \left|\sum_{\nu\geq0}e^{-2\beta t(\nu+1)}\,2^{-2\nu\theta}\,\int\frac{2}{\sqrt{a_{2^{-\nu}}}}\,\,\Re\left([\vphi_{\nu}(D_x),b_1]\partial_xu\,
\overline{\left(\partial_tu_\nu+\frac{\partial_t\sqrt{a_{2^{-\nu}}}}{2\sqrt{a_{2^{-\nu}}}}u_\nu\right)}
\right)dx\right|\,\leq \\
& & \qquad\qquad\qquad\qquad\qquad\qquad\qquad\qquad\qquad\leq\,C\,\wtilde{M}(\omega,\theta)
\sum_{\nu=0}^{+\infty}(\nu+1)\,e^{-2\beta(\nu+1)t}\,2^{-2\nu\theta}\,e_{\nu,2^{-\nu}}(t)\;.
\end{eqnarray*}

\subsubsection{Term with $[\vphi_\nu(D_x),c]$}

Finally, we have to deal with the commutator $[\vphi_\nu(D_x),c(t,x)]$.

First, observe that there exist constants, depending only on $\lambda_0$, $\Lambda_0$ and $C_0$ as usual, such that one has
\begin{eqnarray*}
& & \left|\int\frac{2}{\sqrt{a_{2^{-\nu}}}}\,\,\Re\left([\vphi_{\nu}(D_x),c]\partial_xu\,
\overline{\left(\partial_tu_\nu+\frac{\partial_t\sqrt{a_{2^{-\nu}}}}{2\sqrt{a_{2^{-\nu}}}}u_\nu\right)}
\right)\,dx\right|\,\leq \\
& & \qquad\qquad\qquad\qquad\qquad\leq\,C\,\|[\vphi_{\nu}(D_x),c]u\|_{L^2}\,\left\|\frac{1}{\sqrt[4]{a_{2^{-\nu}}}}\left(\partial_tu_\nu+
\frac{\partial_t\sqrt{a_{2^{-\nu}}}}{2\sqrt{a_{2^{-\nu}}}}u_\nu\right)\right\|_{L^2} \\
& & \qquad\qquad\qquad\qquad\qquad\leq\,C\,\sum_{\mu\geq0}\|[\vphi_\nu(D_x),c]\Psi_\mu\|_{\mc{L}(L^2)}\,\|u_\mu\|_{L^2}\,
\left(e_{\nu,2^{-\nu}}(t)\right)^{1/2}\\
& & \qquad\qquad\qquad\qquad\qquad\leq\,2\,C\,\sum_{\mu\geq0}\|[\vphi_\nu(D_x),c]\Psi_\mu\|_{\mc{L}(L^2)}\,
2^{-\mu}\,\|D_xu_\mu\|_{L^2}\,\left(e_{\nu,2^{-\nu}}(t)\right)^{1/2}\\
& & \qquad\qquad\qquad\qquad\qquad\leq\,2\,C\,\sum_{\mu\geq0}\|[\vphi_\nu(D_x),c]\Psi_\mu\|_{\mc{L}(L^2)}\,2^{-\mu}\,
\left(e_{\mu,2^{-\mu}}(t)\right)^{1/2}\,\left(e_{\nu,2^{-\nu}}(t)\right)^{1/2}\;.
\end{eqnarray*}

Thereby, we get the estimate
\begin{eqnarray*}
& & \left|\sum_{\nu\geq0}e^{-2\beta t(\nu+1)}\,2^{-2\nu\theta}\,\int\frac{2}{\sqrt{a_{2^{-\nu}}}}\,\,\Re\left([\vphi_{\nu}(D_x),c]\partial_xu
\overline{\left(\partial_tu_\nu+\frac{\partial_t\sqrt{a_{2^{-\nu}}}}{2\sqrt{a_{2^{-\nu}}}}u_\nu\right)}
\right)dx\right|\,\leq \\
& & \qquad\qquad\qquad\leq\,2\,C\sum_{\nu,\mu\geq0}e^{-\beta t(\nu+1)}\,2^{-\nu\theta}\,\left(e_{\nu,2^{-\nu}}(t)\right)^{1/2}\,
e^{-\beta t(\mu+1)}\,2^{-\mu\theta}\,\left(e_{\mu,2^{-\mu}}(t)\right)^{1/2}\,m_{\nu\mu}\;,
\end{eqnarray*}
where we have defined
$$
m_{\nu\mu}\,=\,e^{-(\nu-\mu)\beta t}\,2^{-(\nu-\mu)\theta}\,2^{-\mu}\,\|[\vphi_\nu(D_x),c]\Psi_\mu\|_{\mc{L}(L^2)}\;.
$$

The kernel of the operator $[\vphi_\nu(D_x),c]$ is
$$
n'(x,y)\,=\,\what{\psi}(2^\nu(x-y))\,2^{\nu}\,(c(t,y)-c(t,x))\,;
$$
so, remembering that $c$ is bounded over $\mbb{R}\times\mbb{R}$, from Schur's lemma one gets
$$
\|[\vphi_\nu(D_x),c]\|_{\mc{L}(L^2)}\,\leq\,C\qquad\qquad\forall\,\nu\geq0\,,
$$
where the constant $C$ depends only on $\|c\|_{L^\infty(\mbb{R}_t\times\mbb{R}_x)}$.

Again, we are going to estimate the kernel $m_{\nu\mu}$ to apply Schur's lemma.
\begin{enumerate}
 \item First, we take $\mu\leq2$ and we have
\begin{eqnarray*}
\sum_{\nu\geq0}|m_{\nu\mu}| & \leq &  C\,e^{(\mu+1)\beta t}\,2^{(\mu+1)\theta}\,2^{-\mu}
\sum_{\nu\geq0}e^{-(\nu+1)\beta t}\,2^{-(\nu+1)\theta} \\
& \leq & C\,e^{3\beta t}\,2^{3\theta}\,\sum_{\nu\geq0}e^{-(\nu+1)(\beta t+\theta\log2)}\,(\nu+1)^{1/2} \\
& \leq & C\,e^{3\beta t}\,2^{3\theta}\,\alpha_2(\beta t+\theta\log2) \\
& \leq & C\,2^{\frac{9}{2}\theta}\,\alpha_2(\theta\log2)\;.
\end{eqnarray*}
 \item Now, we fix $\mu\geq3$ and we consider the first part of the series:
\begin{eqnarray*}
\sum_{\nu=0}^{\mu-3}|m_{\nu\mu}| & \leq & C\,e^{(\mu+1)\beta t}\,2^{-(\mu+1)(1-\theta)}\,2^{-\mu}\,
\sum_{\nu=0}^{\mu-3} e^{-(\nu+1)\beta t}\,2^{(\nu+1)(1-\theta)}\,2^{\mu-\nu}\frac{(\nu+1)^{1/2}}{(\nu+1)^{1/2}} \\
& \leq & C\,e^{(\mu+1)\beta t}\,2^{-(\mu+1)(1-\theta)}\,(\mu-2)^{1/2}\,
\sum_{\nu=0}^{\mu-3} e^{(\nu+1)((1-\theta)\log2-\beta t)}\,(\nu+1)^{-1/2} \\
& \leq & C\,e^{(\mu+1)\beta t}\,2^{-(\mu+1)(1-\theta)}\,(\mu-2)^{1/2}\,
\alpha_1((1-\theta)\log2-\beta t)\,\cdot \\
& & \qquad\qquad\qquad\qquad\qquad\qquad\qquad\qquad\cdot\,e^{-(\mu-2)\beta t}\,2^{(\mu-2)(1-\theta)}\,(\mu-2)^{-1/2} \\
& \leq & C\,e^{3\beta t}\,2^{3\theta}\,\alpha_1\left(\left(1-\frac{3}{2}\theta\right)\log2\right) \\
& \leq & C\,2^{\frac{9}{2}\theta}\,\alpha_1\left(\left(1-\frac{3}{2}\theta\right)\log2\right)\;.
\end{eqnarray*}
For the second part, one has:
\begin{eqnarray*}
\sum_{\nu=\mu-2}^{+\infty}|m_{\nu\mu}| & \leq &  C\,e^{(\mu+1)\beta t}\,2^{(\mu+1)\theta}\,2^{-\mu}
\sum_{\nu=\mu-2}^{+\infty}e^{-(\nu+1)\beta t}\,2^{-(\nu+1)\theta}\frac{(\nu+1)^{1/2}}{(\nu+1)^{1/2}} \\
& \leq & C\,e^{(\mu+1)\beta t}\,2^{(\mu+1)\theta}\,2^{-\mu}\,\alpha_2(\beta t+\theta\log2)\,
e^{-(\mu-1)\beta t}\,2^{-(\mu-1)\theta} \\
& \leq & C\,2^{3\theta}\,\alpha_2(\theta\log2)\;.
\end{eqnarray*}
 \item Now, we fix $\nu\geq0$. Initially, we consider
\begin{eqnarray*}
\sum_{\mu=0}^{\nu+2}|m_{\nu\mu}| & \leq & C\,e^{-(\nu+1)\beta t}\,2^{-(\nu+1)\theta}\,
\sum_{\mu=0}^{\nu+2}e^{(\mu+1)(\beta t+\theta\log2)}\,2^{-\mu}\,\frac{(\mu+1)^{1/2}}{(\mu+1)^{1/2}} \\
& \leq & C\,e^{-(\nu+1)\beta t}\,2^{-(\nu+1)\theta}\,(\nu+3)^{1/2}\,\alpha_1(\beta t+\theta\log2)\,\cdot \\
& & \qquad\qquad\qquad\qquad\qquad\qquad\cdot\,e^{(\nu+3)\beta t}\,2^{(\nu+3)\theta}\,(\nu+3)^{-1/2} \\
& \leq & C\,e^{2(\beta t+\theta\log2)}\,\alpha_1(\theta\log2) \\
& \leq & C\,2^{3\theta}\,\alpha_1(\theta\log2)\;.
\end{eqnarray*}
The second part of the series, instead, can be treated as follow:
\begin{eqnarray*}
\sum_{\mu=\nu+3}^{+\infty}|m_{\nu\mu}| & \leq & C\,e^{-(\nu+1)\beta t}\,2^{(\nu+1)(1-\theta)}\,
\sum_{\mu=\nu+3}^{+\infty}e^{(\mu+1)\beta t}\,2^{-\mu}\,2^{-(\mu+1)(1-\theta)}\,2^{\mu-\nu}\,
\frac{(\mu+1)^{\frac{1}{2}}}{(\mu+1)^{\frac{1}{2}}} \\
 & \leq & C\,e^{-(\nu+1)\beta t}\,2^{(\nu+1)(1-\theta)}\,(\nu+4)^{-\frac{1}{2}}\,
\sum_{\mu=\nu+3}^{+\infty}e^{-(\mu+1)((1-\theta)\log2-\beta t)}\,(\mu+1)^{\frac{1}{2}} \\
& \leq & C\,e^{-(\nu+1)\beta t}\,2^{(\nu+1)(1-\theta)}\,\alpha_2((1-\theta)\log2-\beta t)\,e^{(\nu+4)\beta t}\,2^{-(\nu+4)(1-\theta)} \\
& \leq & C\,2^{\frac{9}{2}\theta}\,\alpha_2\left(\left(1-\frac{3}{2}\theta\right)\log2\right)\;.
\end{eqnarray*}
\end{enumerate}
Finally, we obtain:
\begin{eqnarray*}
& & \left|\sum_{\nu\geq0}e^{-2\beta t(\nu+1)}\,2^{-2\nu\theta}\,\int\frac{2}{\sqrt{a_{2^{-\nu}}}}\,\,\Re\left([\vphi_{\nu}(D_x),c]\partial_xu\,
\overline{\left(\partial_tu_\nu+\frac{\partial_t\sqrt{a_{2^{-\nu}}}}{2\sqrt{a_{2^{-\nu}}}}u_\nu\right)}
\right)dx\right|\,\leq \\
& & \qquad\qquad\qquad\qquad\qquad\qquad\qquad\qquad\qquad\leq\,C\,\varPi(\theta)\,
\sum_{\nu=0}^{+\infty}(\nu+1)\,e^{-2\beta(\nu+1)t}\,2^{-2\nu\theta}\,e_{\nu,2^{-\nu}}(t)\;,
\end{eqnarray*}
where the function $\varPi$ is the same used in the estimate of the term $[\vphi_\nu(D_x),a]$.

\subsection{End of the proof of theorem \ref{(N)teor_teorema1}}

Now we can complete the proof of theorem \ref{(N)teor_teorema1}.

First, remembering the definition of total energy given by (\ref{(N)eq_E_tot}), we have that there exists a constant $C>0$, depending only on
$\theta$, such that
$$
\left|\sum_{\nu=0}^{+\infty}e^{-\beta(\nu+1)t}\,2^{-2\nu\theta}\,\int\frac{2}{\sqrt{a_\veps}}\,\,\Re
\left((Lu)_\nu\,\cdot\,\overline{\left(\partial_tu_\nu+
\frac{\partial_t\sqrt{a_\veps}}{2\sqrt{a_\veps}}u_\nu\right)}\right)\,dx\right|\,\leq\,C\,(E(t))^{1/2}\,\|Lu\|_{H^{-\theta-\beta^*t}}\,.
$$
Therefore, if we set $\wtilde{\varPi}(\omega,\theta)=\max\{\wtilde{M}(\omega,\theta),\varPi(\theta)\}$, from relation (\ref{(N)eq_dt_e_nu,eps_II})
and from the estimates proved in the previous subsection, we have that, for suitable constants, depending only on $\lambda_0$, $\Lambda_0$, $C_0$
and on the norms of the coefficients of operator $L$ in their respective functional spaces, the following inequality holds:
\begin{eqnarray*}
\frac{d}{dt}E(t) & \leq & \left(C+C'\,\wtilde{\varPi}(\omega,\theta)-2\beta\right)\,
\sum_{\nu=0}^{+\infty}\,(\nu+1)\,e^{-2\beta(\nu+1)t}\,2^{-2\nu\theta}\,e_{\nu,2^{-\nu}}(t) \\
 & + & C''\,(E(t))^{1/2}\,\|Lu\|_{H^{-\theta-\beta^*t}}\,.
\end{eqnarray*}
Now, let us fix $\beta$ large enough, such that $C+C'\,\wtilde{\varPi}(\omega,\theta)-2\beta\leq0$: we can always do this, on condition that we take
$T$ small enough. With this choice, we have
$$
\frac{d}{dt}E(t)\,\leq\,C''\,(E(t))^{1/2}\,\|Lu\|_{H^{-\theta-\beta^*t}}\;;
$$
now, the thesis of the theorem follows from Gronwall's lemma and remark \ref{remark:E(0)-E(t)}.


\end{document}